\documentclass{amsart}
\usepackage {amsmath, amscd}
\usepackage {amssymb}
\usepackage {epsfig}
\usepackage {bm}
\usepackage {indentfirst} 
\usepackage{color}
\usepackage{tikz-cd}

\usepackage[active]{srcltx}

\numberwithin{equation}{section}

\def\<{\langle}
\def\>{\rangle}

\def\LL{{\mathcal L}}

\def\1{\mathbf{1}}

\newtheorem{lem}{Lemma}[section]
\newtheorem{prop}[lem]{Proposition}
\newtheorem{thm}[lem]{Theorem}
\newtheorem{corollary}[lem]{Corollary}

\theoremstyle{definition}

\newtheorem{example}[lem]{Example}

\title{Generalized Crofoot transform and Matrix valued asymmetric truncated Toeplitz operators  }

\author{Rewayat Khan}
\address{Abdus Salam School of Mathematical Sciences, GC University, Lahore, Pakistan}
\email{rewayat.khan@gmail.com}


\begin{document}
	\maketitle

\begin{abstract}
	 Matrix valued asymmetric truncated Toeplitz operator are compression of multiplication operators acting between two model spaces. These are the generalization of matrix valued truncated Toeplitz operators. In this paper we use generalized Crofoot transform to characterize symbol of matrix valued asymmetric truncated Toeplitz operator equal to zero operator.
\end{abstract}

\section{Introduction}
Truncated Toeplitz operators are compression of multiplication operators to the backward shift invariant subspace of the Hardy-Hilbert space called model space denoted by $K_{\theta}=[\theta H^{2}]^{\perp}$ with $\theta$ a complex-valued inner function i.e $|\theta(e^{it})|=1$ a.e on $\mathbb{T}$. The study of truncated Toeplitz operators was initiated by D. Sarason in 2007 (see \cite{sar}).\par
In \cite{RDK}, we have introduced the study of matrix valued truncated Toeplitz operators which makes sense by considering $E$ a finite $d$-dimensional Hilbert space and $\Theta$ an $\mathcal{L}(E)$-valued inner function. The corresponding model space is denoted by $K_{\Theta}=[\Theta H^{2}(E)]^{\perp}$. This is the main instinct of study in our paper. One of the classical result in \cite{RDK} which characterize the symbol of the matrix valued truncated Toeplitz operators to be zero operator. Also like the truncated Toeplitz operators, the matrix valued truncated Toeplitz operators is not uniquely determined by its symbol (see \cite{RDK}, Theorem 6.3 and Corollary 6.4). \par
Recently in \cite{RK}, we have introduced a generalization of the Crofoot transform, the so-called generalized Crofoot transform,  which is a unitary operator maps one model space onto another model space.\par
The study of truncated Toeplitz operators has a natural generalization
to the study of new class of operators, so-called  asymmetric truncated Toeplitz operators. This study is recently initiated in \cite{cj, r, cj2}(see also \cite{L}, \cite{l2} and \cite{JL}). By the same way the study of matrix valued truncated Toeplitz operators can be extended to the study of matrix valued asymmetric truncated Toeplitz operators which acts between two different model spaces. In \cite{cj}, zero asymmetric truncated Toeplitz operators has been characterized in terms of its symbol for a special case when one of the inner function divide the other (see Theorem 4.4). But more general proof can be found in \cite{JL} (see Theorem 2.1). Asymmetric truncated Toeplitz operators on finite dimensional model spaces is defined in \cite{jl}.

 The structure of the paper is the following. In Section 2, we gather some properties of the model spaces, bounded truncated and matrix valued truncated Toeplitz operators. We use generalized Crofoot transform to characterize the symbol of the matrix valued asymmetric truncated Toeplitz operators in Section 3. In section 4, we characterize zero matrix valued asymmetric truncated Toeplitz operator, in terms of its symbol.

\section{Preliminaries}

Let $\mathbb{C}$ denote the complex plane, $\mathbb{D}={\{z\in\mathbb{C}: |z|<1}\}$ the unit disc, $\mathbb{T}={\{z\in\mathbb{C}:|z|=1}\}$ the unit circle.
Throughout the paper $E$ will denote $d$-dimensional complex Hilbert space, and $\LL(E)$ the algebra of bounded linear operators on $E$, which  may be identified with $d\times d$ matrices.

The space $L^{2}(E)$  is defined,  as usual, by
\[
L^{2}(E)={\Big\{f:\mathbb{T}\to E: f(e^{it})=\sum\limits_{-\infty}^{\infty}a_{n}e^{int}:  a_{n}\in E,\quad   \sum\limits_{-\infty}^{\infty}\|a_{n}\|^{2}<\infty }\Big\},
\]
endowed with the inner product
\begin{equation*}
\langle f,g\rangle_{L^{2}(E)}=\frac{1}{2\pi}\int\limits_{0}^{2\pi}\langle f(e^{it}),g(e^{it})\rangle_{E}\,dt.
\end{equation*}
As usual, $H^{2}(E)$ is the space of $E$-valued analytic functions defined on the unit disc $\mathbb{D}$, whose Taylor coefficients are square summable. It can also be seen as a closed subspace of $L^{2}(E)$ which consists of functions from $L^{2}(E)$ such that their Fourier coefficients corresponding to negative indices vanishes.\par
The unilateral shift $S$ on $H^{2}(E)$ is the operator of multiplication by $z$, that is
$$Sf=zf(z).$$
The adjoint $S^{*}$ is called the backward shift and is given by the formula
$$S^{*}f=\frac{1}{z}(f(z)-f(0)).$$
A nonconstant function  $\Theta\in H^{\infty}(\mathcal{L}(E))$ is called inner function if $\Theta$ is an isometry a.e on $\mathbb{T}$. In sequel we will always suppose that $\Theta$ is pure, which means that $\|\Theta(0)\|<1$.
Beurling theorem provides a characterization of all backward shift nontrivial closed invariant subspace of $H^{2}(E)$ that is, a closed nontrivial subspace of $H^{2}(E)$ is $S^{*}$-invariant if and only if it is of the form
$$K_{\Theta}=H^{2}(E)\ominus \Theta H^{2}(E),$$
 the space $K_{\Theta}$ is called model space.\par
A Toeplitz operator $T_{\Phi}$ with symbol $\Phi\in L^{\infty}(\mathcal{L}(E))$ is defined on $H^{2}(E)$ by
$$T_{\Phi}f=P(\Phi f).$$ Clearly $T_{\Phi}$ is bounded if and only if $\Phi\in L^{\infty}(\mathcal{L}(E))$. The operators $S$ and $S^{*}$ are example of Toeplitz operators with symbol $zI_{E}$ and $\overline{z}I_{E}$ respectively, where $I_{E}$ is the identity operator on $E$.\par
Recall that model spaces are reproducing kernel Hilbert spaces with reproducing kernel function $k_{\lambda}^{\Theta}x$. It means that for every $f\in K_{\Theta}$ and each $\lambda\in \mathbb{D}$, it satisfies the relation
$$\langle f, k_{\lambda}^{\Theta}x\rangle=\langle f(\lambda), x\rangle,$$
 where the reproducing kernel function $k_{\lambda}^{\Theta}x$ with $x\in E$ is of the form
 $$k_{\lambda}^{\Theta}x=\frac{1}{1-\overline{\lambda}z}(I-\Theta(z)\Theta(\lambda)^{*})x.$$ It is well known that $k_{\lambda}^{\Theta}x\in H^{\infty}(E)$ and the set $K_{\Theta}^{\infty}=K_{\Theta}\cap H^{\infty}(E)$ is dense in $K_{\Theta}$.\par
 In \cite{RDK}, we have developed the study of matrix valued truncated Toeplitz operators and its properties.
 Suppose $\Theta$ is fixed pure inner function and that $\Phi\in L^{2}(\mathcal{L}(E))$. Let $P_{\Theta}$ be orthogonal projection from $H^{2}(E)$ onto $K_{\Theta}$. Then consider the linear map $f\to P_{\Theta}(\Phi f)$ defined on $K_{\Theta}^{\infty}$. Incase it is bounded, it uniquely determines an operator in $\mathcal{L}(K_{\Theta})$, denoted by $A_{\Phi}^{\Theta}$, and is called a matrix valued truncated Toeplitz operator. $\Phi$ is then called symbol of the $A_{\Phi}^{\Theta}$. The space of all matrix valued truncated Toeplitz operators is denoted by $\mathcal{T}(\Theta)$ . It is immediate that $A_{\Phi}^{*}=A_{\Phi^{*}}$. \par
 Recently, the study of the class of standard asymmetric truncated Toeplitz operators is initiated in \cite{cj, r, cj2}. Let $\Theta_{1}$ and $\Theta_{2}$ be two inner functions. A matrix valued asymmetric truncated Toeplitz operator $A_{\Phi}^{\Theta_{1}, \Theta_{2}}$ with symbol $\Phi\in L^{2}(\mathcal{L}(E))$ is an operator from $K_{\Theta_{1}}$ into $K_{\Theta_{2}}$ densely defined by
 $$A_{\Phi}^{\Theta_{1}, \Theta_{2}}f=P_{\Theta_{2}}(\Phi f),$$ for all $f$ in  $K_{\Theta_{1}}^{\infty}$.\par The space of all matrix valued asymmetric truncated Toeplitz operators is denoted by $\mathcal{T}(\Theta_{1}, \Theta_{2})$.

\section {Generalized Crofoot transform and matrix valued asymmetric truncated Toeplitz operators}\label{se:GCT}
Let $W_{1}, W_{2}\in \mathcal{L}(E)$ be strict contraction and $\Theta_{1},\Theta_{2} \in H^{\infty} (\mathcal{L}(E))$ be pure inner functions. Then by \cite{RK} the function $\Theta_{1}^{\prime}$ defined in terms of $W_{1}$ and $\Theta_{1}$ given by
\begin{equation}\label{Eq: 1 inner}
\Theta_{1}^{\prime}(\lambda)=-W_{1}+D_{W_{1}^{*}}(I-\Theta_{1}(\lambda)W_{1}^{*})^{-1}\Theta_{1}(\lambda)D_{W_{1}},
\end{equation}
 is pure inner function.
Similarly, we define
\begin{equation}\label{Eq: 2 inner}
\Theta_{2}^{\prime}(\lambda)=-W_{2}+D_{W_{2}^{*}}(I-\Theta_{2}(\lambda)W_{2}^{*})^{-1}\Theta_{2}(\lambda)D_{W_{2}},
\end{equation}
 is also pure. \par
Let $K_{\Theta_{1}}$ and $K_{\Theta_{2}}$ be the model spaces corresponding to $\Theta_{1}$ and $\Theta_{2}$ respectively and $K_{\Theta_{1}^{\prime}}$ and $K_{\Theta_{2}^{\prime}}$ be the model spaces corresponding to $\Theta_{1}^{\prime}$ and $\Theta_{2}^{\prime}$ respectively.\par
The generalized Crofoot transform is defined in \cite{RK}(see Theorem 3.3) which is a unitary operator $J_{W_{1}}: K_{\Theta_{1}}\longrightarrow K_{\Theta_{1}^{\prime}}$ defined by
$$J_{W_{1}}f=D_{W_{1}^{*}}(I-\Theta_{1}(\lambda)W_{1}^{*})^{-1}f,$$
and its adjoint $J_{W_{1}}^{*}: K_{\Theta_{1}^{\prime}}\longrightarrow K_{\Theta_{1}}$ is given by
$$J_{W_{1}}^{*}g=D_{W_{1}^{*}}(I+\Theta^{\prime}_{1}(\lambda)W_{1}^{*})^{-1}g.$$
 We define $J_{W_{2}}:K_{\Theta_{2}}\to K_{\Theta_{2}^{\prime}}$ and its adjoint $J_{W_{2}}^{*}$ analogously.\par The action of $J_{W_{2}}$ on the reproducing kernel is given in \cite{RK} (see Proposition 3.4) that is
\begin{equation}\label{eq:J(repro)}
J_{W_{2}}(k_{\lambda}^{\Theta_{2}}(I-W_{2}\Theta_{2}(\lambda)^{*})^{-1}D_{W_{2}^{*}}y)=k_{\lambda}^{\Theta_{2}^{\prime}}y,
\end{equation}
\begin{thm}\label{Thm of Crofoot and asymmetric tto}
Let $\Theta_{1}$ and $\Theta_{2}$ be nonconstant inner functions belong to $\mathcal{L}(E)$. An operator $A_{\Phi}^{\Theta_{1}, \Theta_{2}}$ with $\Phi\in L^{2}(\mathcal{L}(E))$ belong to $\mathcal{T}(\Theta_{1}, \Theta_{2})$ if and only if $J_{W_{2}}A_{\Phi}^{\Theta_{1}, \Theta_{2}}J_{W_{1}}^{*}=A_{\Psi}^{\Theta_{1}^{\prime}, \Theta_{2}^{\prime}}\in \mathcal{T}(\Theta_{1}^{\prime}, \Theta_{2}^{\prime})$, where
\begin{equation}
\Psi = D_{W_{2}^{*}}(I-\Theta_{2}(\lambda)W_{2}^{*})^{-1}\Phi D_{W_{1}^{*}}(I+\Theta^{\prime}_{1}(\lambda)W_{1}^{*})^{-1}.
\end{equation}
\end{thm}
\begin{proof}
Let $f\in K_{\Theta_{1}^{\prime}}$ then consider
\begin{eqnarray*}
\langle J_{W_{2}}A_{\Phi}^{\Theta_{1}, \Theta_{2}}J_{W_{1}}^{*}f, k_{\lambda}^{\Theta_{2}^{\prime}}x\rangle
&=&\langle J_{W_{2}}A_{\Phi}^{\Theta_{1}, \Theta_{2}}D_{W_{1}^{*}}(I+\Theta^{\prime}_{1}(\lambda)W_{1}^{*})^{-1}f, k_{\lambda}^{\Theta_{2}^{\prime}}x\rangle\\
&=&\langle P_{\Theta_{2}}(\Phi D_{W_{1}^{*}}(I+\Theta^{\prime}_{1}(\lambda)W_{1}^{*})^{-1}f), J_{W_{2}}^{*}k_{\lambda}^{\Theta_{2}^{\prime}}x\rangle\\
&=&\langle P_{\Theta_{2}}(\Phi D_{W_{1}^{*}}(I+\Theta^{\prime}_{1}(\lambda)W_{1}^{*})^{-1}f), k_{\lambda}^{\Theta_{2}}(I-W_{2}\Theta_{2}(\lambda)^{*})^{-1}D_{W_{2}^{*}}x\rangle \\
&=&\langle \Phi D_{W_{1}^{*}}(I+\Theta^{\prime}_{1}(\lambda)W_{1}^{*})^{-1}f), P_{\Theta_{2}}(k_{\lambda}^{\Theta_{2}}(I-W_{2}\Theta_{2}(\lambda)^{*})^{-1}D_{W_{2}^{*}}x)\rangle \\
&=&\langle \Phi D_{W_{1}^{*}}(I+\Theta^{\prime}_{1}(\lambda)W_{1}^{*})^{-1}f, k_{\lambda}^{\Theta_{2}}(I-W_{2}\Theta_{2}(\lambda)^{*})^{-1}D_{W_{2}^{*}}x\rangle \\
&=&\langle J_{W_{2}} (\Phi D_{W_{1}^{*}}(I+\Theta^{\prime}_{1}(\lambda)W_{1}^{*})^{-1}f), J_{W_{2}}(k_{\lambda}^{\Theta_{2}}(I-W_{2}\Theta_{2}(\lambda)^{*})^{-1}D_{W_{2}^{*}}x)\rangle \\
&=&\langle D_{W_{2}^{*}}(I-\Theta_{2}(\lambda)W_{2}^{*})^{-1}\Phi D_{W_{1}^{*}}(I+\Theta^{\prime}_{1}(\lambda)W_{1}^{*})^{-1}f, k_{\lambda}^{\Theta^{\prime}_{2}}x\rangle \\
&=&\langle P_{\Theta_{2}^{\prime}}(\Psi f), k_{\lambda}^{\Theta^{\prime}_{2}}x\rangle=\langle A^{\Theta_{1}^{\prime}, \Theta_{2}^{\prime}}_{\Psi}f, k_{\lambda}^{\Theta^{\prime}_{2}}x\rangle.
\end{eqnarray*}
Therefore
$$\langle J_{W_{2}}A_{\Phi}^{\Theta_{1}, \Theta_{2}}J_{W_{1}}^{*}f, g\rangle=\langle A^{\Theta_{1}^{\prime}, \Theta_{2}^{\prime}}_{\Psi}f, g\rangle$$
for any $g$ in the linear span of $k_{\lambda}^{\Theta^{\prime}_{2}}x$, $\lambda\in \mathbb{D}$ and $x\in E$. The required result follows by the density of the last set in $K_{\Theta_{2}^{\prime}}$.\par
 The symbol $\Psi$ is given by
$$\Psi= D_{W_{2}^{*}}(I-\Theta_{2}(\lambda)W_{2}^{*})^{-1}\Phi D_{W_{1}^{*}}(I+\Theta^{\prime}_{1}(\lambda)W_{1}^{*})^{-1}.$$
Now let $A_{\Psi}^{\Theta_{1}^{\prime}, \Theta_{2}^{\prime}}\in \mathcal{T}(\Theta_{1}^{\prime},\Theta_{2}^{\prime})$,  $f\in K_{\Theta_{1}}$ and $x=(I-W_{2}\Theta_{2}(\lambda)^{*})^{-1}D_{W_{2}^{*}}y$ then
\begin{eqnarray*}
\langle J_{W_{2}}^{*}A_{\Psi}^{\Theta_{1}^{\prime}, \Theta_{2}^{\prime}}J_{W_{1}}f, k_{\lambda}^{\Theta_{2}}x \rangle
&=&\langle A_{\Psi}^{\Theta_{1}^{\prime}, \Theta_{2}^{\prime}}J_{W_{1}}f, J_{W_{2}}k_{\lambda}^{\Theta_{2}}x \rangle\\
&=&\langle A_{\Psi}^{\Theta_{1}^{\prime}, \Theta_{2}^{\prime}}J_{W_{1}}f, k_{\lambda}^{\Theta^{\prime}_{2}}y \rangle\\
&=&\langle P_{ \Theta_{2}^{\prime}}(\Psi D_{W_{1}^{*}}(I-\Theta_{1}(\lambda)W_{1}^{*})^{-1}f), k_{\lambda}^{\Theta^{\prime}_{2}}y \rangle\\
&=&\langle \Psi D_{W_{1}^{*}}(I-\Theta_{1}(\lambda)W_{1}^{*})^{-1}f, k_{\lambda}^{\Theta^{\prime}_{2}}y \rangle\\
&=&\langle J_{W_{2}}^{*} (\Psi D_{W_{1}^{*}}(I-\Theta_{1}(\lambda)W_{1}^{*})^{-1}f),  J_{W_{2}}^{*}k_{\lambda}^{\Theta^{\prime}_{2}}y \rangle\\
&=&\langle D_{W_{2}^{*}}(I+\Theta_{2}^{\prime}(\lambda)W_{2}^{*})^{-1} \Psi D_{W_{1}^{*}}(I-\Theta_{1}(\lambda)W_{1}^{*})^{-1}f, k_{\lambda}^{\Theta_{2}}x \rangle\\
&=&\langle P_{\Theta_{2}}(\Phi f), k_{\lambda}^{\Theta_{2}}x\rangle
\end{eqnarray*}
where
\begin{equation}
\Phi=D_{W_{2}^{*}}(I+\Theta_{2}^{\prime}(\lambda)W_{2}^{*})^{-1} \Psi D_{W_{1}^{*}}(I-\Theta_{1}(\lambda)W_{1}^{*})^{-1}.
\end{equation}
\end{proof}
\begin{example}Let $\Theta_1(z)=z^nI_{E},\Theta_2(z)=z^mI_{E}$, such that $m<n$,  then the corresponding
model spaces are
 $$K_{\Theta_1}=\left\lbrace a_0+a_1z+...+a_{n-1}z^{n-1} \right\rbrace, K_{\Theta_2}=\left\lbrace b_0+b_1z+...+b_{m-1}z^{m-1}   \right\rbrace,$$
 and $\Phi(e^{it})=\sum_{s\in \mathbb{Z}}e^{it}\Delta_s$, with $\Delta_s\in\mathcal{L}(E)$, the operator
 $A_\Phi^{{\Theta_1},{\Theta_2}}\in\mathcal{T}(\Theta_{1}, \Theta_{2})$ has the block matrix decomposition

  $$ A_\Phi^{{\Theta_1},{\Theta_2}}=\left(
 \begin{array}{cccccc}

\Delta_0 & \Delta_1& \dots&\Delta_{m-1}& \dots&\Delta_{n-1} \\
\Delta_{-1}  &  \Delta_0 &\dots &\Delta_{m-2} &\dots&\Delta_{n-2}  \\
     \vdots & \vdots &\vdots & \dots&\vdots &  \vdots\\
\Delta_{-(m-1)} & \Delta_{-(m-2)} &\dots&\Delta_0  &\dots &\Delta_{n-m}\
 \end{array}
 \right),$$
\end{example}

\section{Condition for $A_{\Phi}^{\Theta_{1}, \Theta_{2}}=0$}
Denote by $\mathcal{M}_{\Theta_{1}}$ (respectively by $\mathcal{M}_{\Theta_{2}}$), the orthogonal complement of $\Theta_{1} H^{2}(\mathcal{L}(E))$ (respectively of $\Theta_{2}H^{2}(\mathcal{L}(E))$) in $H^{2}(\mathcal{L}(E))$ endowed with the Hilbert-Schmidt norm.
\begin{lem}\label{Eq: condition for zero oper}
If $\Phi\in[\Theta_{1} H^{2}(\mathcal{L}(E))]^{*}+ \Theta_{2} H^{2}(\mathcal{L}(E))$ then $A_{\Phi}^{\Theta_{1},\Theta_{2}}=0.$
\end{lem}
\begin{proof}
Let $\Phi\in[\Theta_{1} H^{2}(\mathcal{L}(E))]^{*}+ \Theta_{2} H^{2}(\mathcal{L}(E))$ then $\Phi=\Theta_{2}\Phi_{2}+\Phi_{1}^{*}\Theta_{1}^{*}$ with $\Phi_{1}, \Phi_{2}\in H^{2}(\mathcal{L}(E))$ and $f\in K_{\Theta_{1}}$. Then
$$A_{\Phi}^{\Theta_{1}, \Theta_{2}}f=P_{\Theta_{2}}(\Theta_{2}\Phi_{2}f)+P_{\Theta_{2}}(\Phi_{1}^{*}\Theta_{1}^{*}f).$$
It is clear that $P_{\Theta_{2}}(\Theta_{2}\Phi_{2}f)$. On the other hand if $g\in K_{\Theta_{2}}$, then
$$\langle P_{\Theta_{2}}(\Phi_{1}^{*}\Theta_{1}^{*}f), g\rangle=\langle \Phi_{1}^{*}\Theta_{1}^{*}f, g\rangle=\langle f,\Theta_{1}\Phi_{1} g\rangle=0.$$
Therefore we have $P_{\Theta_{2}}\Phi_{1}^{*}\Theta_{1}^{*}f$ is orthogonal to $K_{\Theta_{2}}$. It follows the required result.
\end{proof}

\begin{lem}
Suppose $A_{\Phi}^{\Theta_{1}, \Theta_{2}}=0$ (see \cite{RDK})\label{Eq: Req for APhi equal zero}
\begin{enumerate}
\item  If $\Phi\in H^{2}(\mathcal{L}(E))$ then $\Phi=\Theta_{2} \Phi_{1}$ for some $\Phi_{1}\in H^{2}(\mathcal{L}(E)).$
\item  If $ \Phi^{*}\in H^{2}(\mathcal{L}(E))$ then $\Phi=(\Theta_{2} \Phi_{2})^{*}$ for some $\Phi_{2}\in H^{2}(\mathcal{L}(E)).$
\end{enumerate}
\end{lem}
\begin{prop}\label{Prop: of inclusion}
Let $\Theta_{1}$ and $\Theta_{2}$ be nonconstant inner functions. If
\begin{equation}\label{eq: of inclusion of model space}
K_{\Theta_{1}}\subset \Theta_{2}H^{2}(E),
\end{equation}
 then both $\Theta_{1}$ and $\Theta_{2}$ have no zero in $\mathbb{D}$, or at least one of the function $\Theta_{1}$ or $\Theta_{2}$ is a constant function.
 \end{prop}
 \begin{proof}
 Suppose that (\ref{eq: of inclusion of model space}) holds. If $\Theta_{2}(w_{0})=0$ for some $w_{0}\in \mathbb{D}$, then $f(\omega_{0})=0$ for all $f\in K_{\Theta_{1}}$. For $f=k_{w_{0}}^{\Theta_{1}}x$, where $x\in E$ we get
 $$k_{w_{0}}^{\Theta_{1}}(w_{0})x=\frac{1}{1-|w_{0}|^{2}}(I-\Theta_{1}(w_{0})\Theta_{1}(w_{0})^{*})x=0.$$ Therefore
 $\Theta_{1}(w_{0})\Theta_{1}(w_{0})^{*}=I$, that is $\Theta_{1} (w_{0})$ is unitary on the interior point of $\mathbb{D}$. It implies that $\Theta_{1}$ is constant inner function. Hence it follows that $\Theta_{2}$ has no zero in $\mathbb{D}$. But (\ref{eq: of inclusion of model space}) is equivalent to
 $$K_{\Theta_{2}}\subset \Theta_{1} H^{2}(E),$$
 and by the same argument $\Theta_{1}$ has no zero in $\mathbb{D}$, or $\Theta_{2}$ is constant.
 \end{proof}
Proposition \ref{Prop: of inclusion} can be stated as follows. If $\Theta_{1}$, $\Theta_{2}$ are two nonconstant inner functions and at least one of them has a zero in $\mathbb{D}$, then the inclusion $K_{\Theta_{1}}\subset \Theta_{2} H^{2}(E)$ does not hold.

\begin{thm}\label{Eq: main Theorem}
Let $\Theta_{1}$ and $\Theta_{2}$ be nonconstant inner functions and let $A_{\Phi}^{\Theta_{1},\Theta_{2}}:K_{\Theta_{1}}\longrightarrow K_{\Theta_{2}}$ be a bounded matrix valued asymmetric truncated Toeplitz operators with $\Phi\in L^{2}(\mathcal{L}(E))$. Then $A_{\Phi}^{\Theta_{1}, \Theta_{2}}=0$ if and only if $\Phi \in [\Theta_{1}H^{2}(\mathcal{L}(E))]^{*}+\Theta_{2} H^{2}(\mathcal{L}(E)).$
 \end{thm}
 To prove this theorem we first prove the following Proposition.

 \begin{prop}
 Let $\Theta_{1}$ and $\Theta_{2}$ be nonconstant inner functions such that each of them has a zero in $\mathbb{D}$ and let $A_{\Phi}^{\Theta_{1},\Theta_{2}}:K_{\Theta_{1}}\longrightarrow K_{\Theta_{2}}$ be a bounded matrix valued asymmetric truncated Toeplitz operators with $\Phi\in L^{2}(\mathcal{L}(E))$. Then $A_{\Phi}^{\Theta_{1}, \Theta_{2}}=0$ if and only if $\Phi \in [\Theta_{1}H^{2}(\mathcal{L}(E))]^{*}+\Theta_{2} H^{2}(\mathcal{L}(E)).$
 \end{prop}
\begin{proof}\label{Prop: with condition of zero}
If $\Phi \in [\Theta_{1}H^{2}(\mathcal{L}(E))]^{*}+\Theta_{2} H^{2}(\mathcal{L}(E))$ then $A_{\Phi}^{\Theta_{1}, \Theta_{2}}=0$, follows from  Lemma \ref{Eq: condition for zero oper}. This part of the proof does not depend on the existence of zeros of $\Theta_{1}$ and $\Theta_{2}$. \par
Let us now assume that $A_{\Phi}^{\Theta_{1}, \Theta_{2}}=0$. By the first part of the proof, we can take $\Phi=\Psi_{1}+\Psi_{2}^{*}$ for $\Psi_{1}\in M_{\Theta_{1}}$ and $\Psi_{2}\in M_{\Theta_{2}}$. Let us also assume that $z_{0}\in \mathbb{D}$ is zero of $\Theta_{1}$. Then the reproducing kernel for $K_{\Theta_{1}}$ becomes $$k_{z_{0}}^{\Theta_{1}}x=k_{z_{0}}x=\frac{1}{1-\overline{z_{0}}z}x.$$ Therefore
\begin{eqnarray*}
A_{\Psi_{2}^{*}}^{\Theta_{1}, \Theta_{2}}k_{z_{0}}^{\Theta_{1}}x
&=&P_{\Theta_{2}}(\Psi_{2}^{*}k_{z_{0}}^{\Theta_{1}}x)\\
&=&P_{\Theta_{2}}(\frac{1}{1-\overline{z_{0}}z}(\Psi_{2}^{*}-\Psi_{2}^{*}(z_{0}))x+\frac{1}{1-\overline{z_{0}}z}\Psi_{2}^{*}(z_{0})x)\\
&=&P_{\Theta_{2}}(\frac{\overline{z}}{\overline{z}-\overline{z_{0}}}(\Psi_{2}^{*}-\Psi_{2}^{*}(z_{0}))x+\frac{1}{1-\overline{z_{0}}z}\Psi_{2}^{*}(z_{0})x)\\
&=&\Psi_{2}^{*}(z_{0})k_{z_{0}}^{\Theta_{2}}x,
\end{eqnarray*}
because $\frac{1}{z-z_{0}}(\Psi_{2}-\Psi_{2}(z_{0}))x\in K_{\Theta_{2}}$. \par
Now we have
\begin{eqnarray*}
0=A_{\Phi}^{\Theta_{1}, \Theta_{2}}k_{z_{0}}^{\Theta_{1}}x
&=&A_{\Psi_{1}+\Psi_{2}^{*}}^{\Theta_{1}, \Theta_{2}}k_{z_{0}}^{\Theta_{1}}x\\
&=&A_{\Psi_{1}}k_{z_{0}}^{\Theta_{1}}x+\Psi_{2}^{*}(z_{0})k_{z_{0}}^{\Theta_{2}}x\\
&=&P_{\Theta_{2}}((\Psi_{2}^{*}(z_{0})+\Psi_{1})k_{z_{0}}x),
\end{eqnarray*}
by Proposition \ref{Eq: Req for APhi equal zero} we have
\begin{equation}\label{Eq: For Theta 2 H2}
\Psi_{1}+\Psi_{2}^{*}(z_{0})\in\Theta_{2} H^{2}(\mathcal{L}(E)).
\end{equation}
On the other hand we know that
$$A_{\Psi_{1}^{*}+\Psi_{2}}^{\Theta_{2}, \Theta_{1}}=(A_{\Psi_{1}+\Psi_{2}^{*}}^{\Theta_{1}, \Theta_{2}})^{*}=0,$$ and analogous arguments can be used to show that if some $w_{0}\in \mathbb{D}$ is a zero of $\Theta_{2}$, then
\begin{equation}\label{Eq: For Theta 1 H2}
\Psi_{2}+\Psi_{1}(w_{0})^{*}\in \Theta_{1}H^{2}(\mathcal{L}(E))\quad or \quad \Psi_{2}^{*}+\Psi_{1}(w_{0})\in [\Theta_{1}H^{2}(\mathcal{L}(E))]^{*}
\end{equation}
By (\ref{Eq: For Theta 2 H2}) and (\ref{Eq: For Theta 1 H2}) and by the first part of the proof we obtain
$$A_{\Psi_{1}+\Psi_{2}^{*}(z_{0})+\Psi_{2}^{*}+\Psi_{1}(w_{0})}^{\Theta_{1}, \Theta_{2}}=0,$$
and
$$A_{\Psi_{2}^{*}(z_{0})+\Psi_{1}(w_{0})}^{\Theta_{1}, \Theta_{2}}=-A_{\Psi_{1}+\Psi_{2}^{*}}^{\Theta_{1}, \Theta_{2}}=0,$$
we conclude from this
$$P_{\Theta_{2}}[(\Psi_{2}^{*}(z_{0})+\Psi_{1}(w_{0}))f]=0$$ for all $f\in K_{\Theta_{1}}$.\par
If $\Psi_{2}^{*}(z_{0})+\Psi_{1}(w_{0})\neq0$ then from the last equality we have $P_{\Theta_{2}}(f)=0$ for all $f\in K_{\Theta_{1}}$, which follows that $K_{\Theta_{1}}\subset \Theta_{2} H^{2}(E)$ which is a contradiction. So $$\Psi_{2}^{*}(z_{0})+\Psi_{1}(w_{0})=0$$
and $$\Phi= \Psi_{1}+\Psi_{2}^{*}=\Psi_{1}+\Psi_{2}^{*}(z_{0})+\Psi_{2}^{*}+\Psi_{1}(w_{0})\in [\Theta_{1}H^{2}(\mathcal{L}(E))]^{*}+\Theta_{2}H^{2}(\mathcal{L}(E)).$$
\end{proof}
\begin{proof}[Proof of Theorem \ref{Eq: main Theorem}] The fact that $\Phi\in [\Theta_{1}H^{2}(\mathcal{L}(E))]^{*}+\Theta_{2}H^{2}(\mathcal{L}(E))$ implies $A_{\Phi}^{\Theta_{1}, \Theta_{2}}=0$ was established in Proposition \ref{Eq: condition for zero oper}. \par
Assume now that $\Phi\in L^{2}(\mathcal{L}(E))$ and $A_{\Phi}^{\Theta_{1}, \Theta_{2}}=0$. If $\Theta_{1}(0)=\Theta_{2}(0)=0$, then by Proposition \ref{Eq: condition for zero oper} $\Phi\in [\Theta_{1}H^{2}(\mathcal{L}(E))]^{*}+\Theta_{2}H^{2}(\mathcal{L}(E))$. Assume that $\Theta_{1}(0)\neq 0$ or $\Theta_{2}(0)\neq 0$, put $\Theta_{1}(0)=W_{1}$ and $\Theta_{2}(0)=W_{2}$, then since $\Theta_{1}$ and $\Theta_{2}$ are strict contractions then so are $W_{1}$ and $W_{2}$. By Theorem \ref{Thm of Crofoot and asymmetric tto}
$$A_{\Psi}^{\Theta_{1}^{\prime}, \Theta_{2}^{\prime}}=J_{W_{2}}A_{\Phi}^{\Theta_{1}, \Theta_{2}}J_{W_{1}}^{*}=0,$$
where $$\Psi=D_{W_{2}^{*}}(I-\Theta_{2}(\lambda)W_{2}^{*})^{-1}\Phi D_{W_{1}^{*}}(I+\Theta^{\prime}_{1}(\lambda)W_{1}^{*})^{-1}.$$
A very simple calculation shows that $\Theta_{1}^{\prime}(0)=0$ and $\Theta_{2}^{\prime}(0)=0$, by using Proposition \ref{Eq: condition for zero oper} we have $$\Psi\in [\Theta_{1}^{\prime}H^{2}(\mathcal{L}(E))]^{*}+\Theta_{2}^{\prime} H^{2}(\mathcal{L}(E)).$$ Therefore there exist $\Phi_{1}, \Phi_{2}\in H^{2}(\mathcal{L}(E))$ such that
$$D_{W_{2}^{*}}(I-\Theta_{2}(\lambda)W_{2}^{*})^{-1}\Phi D_{W_{1}^{*}}(I+\Theta^{\prime}_{1}(\lambda)W_{1}^{*})^{-1}=(\Theta_{1}^{\prime}\Phi_{1})^{*}+\Theta_{2}^{\prime}\Phi_{2}$$ and hence we have
\begin{equation}\label{eq: of Phi}
\Phi=(I-\Theta_{2}(\lambda)W_{2}^{*})D_{W_{2}^{*}}^{-1}[(\Theta_{1}^{\prime}\Phi_{1})^{*}+\Theta_{2}^{\prime}\Phi_{2}]
(I+\Theta^{\prime}_{1}(\lambda)W_{1}^{*})D_{W_{1}^{*}}^{-1}.
\end{equation}
After a simple computation we have
$$I+\Theta^{\prime}_{1}(\lambda)W_{1}^{*}=D_{W_{1}^{*}}(I-\Theta_{1}W_{1}^{*})^{-1}D_{W_{1}^{*}}.$$ Therefore (\ref{eq: of Phi}) becomes
\begin{eqnarray*}
\Phi
&=&(I-\Theta_{2}W_{2}^{*})D_{W_{2}^{*}}^{-1}[(\Theta_{1}^{\prime}\Phi_{1})^{*}+\Theta_{2}^{\prime}\Phi_{2}]
(I+\Theta^{\prime}_{1}W_{1}^{*})D_{W_{1}^{*}}^{-1}\\
&=&(I-\Theta_{2}W_{2}^{*})D_{W_{2}^{*}}^{-1}[(\Theta_{1}^{\prime}\Phi_{1})^{*}+\Theta_{2}^{\prime}\Phi_{2}]
D_{W_{1}^{*}}(I-\Theta_{1}W_{1}^{*})^{-1}D_{W_{1}^{*}}D_{W_{1}^{*}}^{-1}\\
&=&(I-\Theta_{2}W_{2}^{*})D_{W_{2}^{*}}^{-1}[(\Theta_{1}^{\prime}\Phi_{1})^{*}+\Theta_{2}^{\prime}\Phi_{2}]
D_{W_{1}^{*}}(I-\Theta_{1}W_{1}^{*})^{-1}\\
&=&(I-\Theta_{2}W_{2}^{*})D_{W_{2}^{*}}^{-1}(\Theta_{1}^{\prime}\Phi_{1})^{*}D_{W_{1}^{*}}(I-\Theta_{1}W_{1}^{*})^{-1}
+(I-\Theta_{2}W_{2}^{*})D_{W_{2}^{*}}^{-1}\Theta_{2}^{\prime}\Phi_{2}
D_{W_{1}^{*}}(I-\Theta_{1}W_{1}^{*})^{-1}\\
&=&D_{W_{2}^{*}}^{-1}(\Theta_{1}^{\prime}\Phi_{1})^{*}D_{W_{1}^{*}}(I-\Theta_{1}W_{1}^{*})^{-1}-\Theta_{2}W_{2}^{*}D_{W_{2}^{*}}^{-1}
(\Theta_{1}^{\prime}\Phi_{1})^{*}D_{W_{1}^{*}}(I-\Theta_{1}W_{1}^{*})^{-1}\\
&+&D_{W_{2}^{*}}^{-1}\Theta_{2}^{\prime}\Phi_{2}
D_{W_{1}^{*}}(I-\Theta_{1}W_{1}^{*})^{-1}
-\Theta_{2}W_{2}^{*}D_{W_{2}^{*}}^{-1}\Theta_{2}^{\prime}\Phi_{2}
D_{W_{1}^{*}}(I-\Theta_{1}W_{1}^{*})^{-1}\\
&=&D_{W_{2}^{*}}^{-1}(\Theta_{1}^{\prime}\Phi_{1})^{*}D_{W_{1}^{*}}(I-\Theta_{1}W_{1}^{*})^{-1}+D_{W_{2}^{*}}^{-1}\Theta_{2}^{\prime}\Phi_{2}
D_{W_{1}^{*}}(I-\Theta_{1}W_{1}^{*})^{-1}\\
&-&\Theta_{2}[W_{2}^{*}D_{W_{2}^{*}}^{-1}(\Theta_{1}^{\prime}\Phi_{1})^{*}D_{W_{1}^{*}}(I-\Theta_{1}W_{1}^{*})^{-1}
+W_{2}^{*}D_{W_{2}^{*}}^{-1}\Theta_{2}^{\prime}\Phi_{2}
D_{W_{1}^{*}}(I-\Theta_{1}W_{1}^{*})^{-1}],
\end{eqnarray*}
therefore we get
\begin{eqnarray*}
\Phi
&=&D_{W_{2}^{*}}^{-1}(\Theta_{1}^{\prime}\Phi_{1})^{*}D_{W_{1}^{*}}(I-\Theta_{1}W_{1}^{*})^{-1}+D_{W_{2}^{*}}^{-1}\Theta_{2}^{\prime}\Phi_{2}
D_{W_{1}^{*}}(I-\Theta_{1}W_{1}^{*})^{-1}\\
&-&\Theta_{2}[W_{2}^{*}D_{W_{2}^{*}}^{-1}(\Theta_{1}^{\prime}\Phi_{1})^{*}D_{W_{1}^{*}}(I-\Theta_{1}W_{1}^{*})^{-1}
+W_{2}^{*}D_{W_{2}^{*}}^{-1}\Theta_{2}^{\prime}\Phi_{2}
D_{W_{1}^{*}}(I-\Theta_{1}W_{1}^{*})^{-1}],
\end{eqnarray*}
\begin{eqnarray*}
D_{W_{2}^{*}}^{-1}(\Theta_{1}^{\prime}\Phi_{1})^{*}D_{W_{1}^{*}}(I-\Theta_{1}W_{1}^{*})^{-1}
&=&D_{W_{2}^{*}}^{-1}(\Phi_{1}^{*}\Theta_{1}^{\prime*})D_{W_{1}^{*}}(I-\Theta_{1}W_{1}^{*})^{-1}\\
&=&D_{W_{2}^{*}}^{-1}[\Phi_{1}^{*}(-W_{1}+D_{W_{1}^{*}}(I-\Theta_{1}W_{1}^{*})^{-1}\Theta_{1}D_{W_{1}})^{*}]D_{W_{1}^{*}}(I-\Theta_{1}W_{1}^{*})^{-1}\\
&=&D_{W_{2}^{*}}^{-1}[\Phi_{1}^{*}(-W_{1}^{*}+D_{W_{1}}\Theta_{1}^{*}(I-W_{1}\Theta_{1}^{*})^{-1}D_{W_{1}^{*}})]
D_{W_{1}^{*}}(I-\Theta_{1}W_{1}^{*})^{-1}\\
&=&D_{W_{2}^{*}}^{-1}[\Phi_{1}^{*}(-W_{1}^{*}D_{W_{1}^{*}}+D_{W_{1}}\Theta_{1}^{*}(I-W_{1}\Theta_{1}^{*})^{-1}D_{W_{1}^{*}}^{2})](I-\Theta_{1}W_{1}^{*})^{-1}\\
&=&D_{W_{2}^{*}}^{-1}[\Phi_{1}^{*}(-D_{W_{1}}W_{1}^{*}+D_{W_{1}}\Theta_{1}^{*}
(I-W_{1}\Theta_{1}^{*})^{-1}D_{W_{1}^{*}}^{2})](I-\Theta_{1}W_{1}^{*})^{-1}\\
&=&D_{W_{2}^{*}}^{-1}[\Phi_{1}^{*}D_{W_{1}}(-W_{1}^{*}+
(I-\Theta_{1}^{*}W_{1})^{-1}\Theta_{1}^{*}D_{W_{1}^{*}}^{2})](I-\Theta_{1}W_{1}^{*})^{-1}\\
&=&D_{W_{2}^{*}}^{-1}\Phi_{1}^{*}D_{W_{1}}(I-\Theta_{1}^{*}W_{1})^{-1}(-(I-\Theta_{1}^{*}W_{1})W_{1}^{*}+
\Theta_{1}^{*}D_{W_{1}^{*}}^{2})(I-\Theta_{1}W_{1}^{*})^{-1}\\
&=&D_{W_{2}^{*}}^{-1}\Phi_{1}^{*}D_{W_{1}}(I-\Theta_{1}^{*}W_{1})^{-1}[-(I-\Theta_{1}^{*}W_{1})W_{1}^{*}+
\Theta_{1}^{*}(I-W_{1}W_{1}^{*})](I-\Theta_{1}W_{1}^{*})^{-1}\\
&=&D_{W_{2}^{*}}^{-1}\Phi_{1}^{*}D_{W_{1}}(I-\Theta_{1}^{*}W_{1})^{-1}[-W_{1}^{*}+\Theta_{1}^{*}W_{1}W_{1}^{*}+
\Theta_{1}^{*}-\Theta_{1}^{*}W_{1}W_{1}^{*}](I-\Theta_{1}W_{1}^{*})^{-1}\\
&=&D_{W_{2}^{*}}^{-1}\Phi_{1}^{*}D_{W_{1}}(I-\Theta_{1}^{*}W_{1})^{-1}[\Theta_{1}^{*}-W_{1}^{*}
](I-\Theta_{1}W_{1}^{*})^{-1}\\
&=&D_{W_{2}^{*}}^{-1}\Phi_{1}^{*}D_{W_{1}}(I-\Theta_{1}^{*}W_{1})^{-1}\Theta_{1}^{*}(I-\Theta_{1}W_{1}^{*})(I-\Theta_{1}W_{1}^{*})^{-1}\\
&=&D_{W_{2}^{*}}^{-1}\Phi_{1}^{*}D_{W_{1}}(I-\Theta_{1}^{*}W_{1})^{-1}\Theta_{1}^{*}\\
&=&[\Theta_{1}(I-W_{1}^{*}\Theta_{1})^{-1}D_{W_{1}}\Phi_{1}D_{W_{2}^{*}}^{-1}]^{*}
\end{eqnarray*}
\begin{equation}\label{Eq: for Theta 1}
D_{W_{2}^{*}}^{-1}(\Theta_{1}^{\prime}\Phi_{1})^{*}D_{W_{1}^{*}}(I-\Theta_{1}W_{1}^{*})^{-1}
=[\Theta_{1}(I-W_{1}^{*}\Theta_{1})^{-1}D_{W_{1}}\Phi_{1}D_{W_{2}^{*}}^{-1}]^{*}.
\end{equation}
and
\begin{eqnarray*}
D_{W_{2}^{*}}^{-1}\Theta_{2}^{\prime}\Phi_{2}
D_{W_{1}^{*}}(I-\Theta_{1}W_{1}^{*})^{-1}
&=&D_{W_{2}^{*}}^{-1}[-W_{2}+D_{W_{2}^{*}}(I-\Theta_{2}W_{2}^{*})^{-1}\Theta_{2}D_{W_{2}}]\Phi_{2}
D_{W_{1}^{*}}(I-\Theta_{1}W_{1}^{*})^{-1}\\
&=&[-D_{W_{2}^{*}}^{-1}W_{2}+(I-\Theta_{2}W_{2}^{*})^{-1}\Theta_{2}D_{W_{2}}]\Phi_{2}
D_{W_{1}^{*}}(I-\Theta_{1}W_{1}^{*})^{-1}\\
&=&-(I-\Theta_{2}W_{2}^{*})^{-1}[(I-\Theta_{2}W_{2}^{*})D_{W_{2}^{*}}^{-1}W_{2}-\Theta_{2}D_{W_{2}}]\Phi_{2}
D_{W_{1}^{*}}(I-\Theta_{1}W_{1}^{*})^{-1}\\
&=&-(I-\Theta_{2}W_{2}^{*})^{-1}[D_{W_{2}^{*}}^{-1}W_{2}-\Theta_{2}W_{2}^{*}D_{W_{2}^{*}}^{-1}W_{2}-\Theta_{2}D_{W_{2}}]\Phi_{2}
D_{W_{1}^{*}}(I-\Theta_{1}W_{1}^{*})^{-1}\\
&=&-(I-\Theta_{2}W_{2}^{*})^{-1}[W_{2}D_{W_{2}}^{-1}-\Theta_{2}W_{2}^{*}W_{2}D_{W_{2}}^{-1}-\Theta_{2}D_{W_{2}}]\Phi_{2}
D_{W_{1}^{*}}(I-\Theta_{1}W_{1}^{*})^{-1}\\
&=&-(I-\Theta_{2}W_{2}^{*})^{-1}[(I-\Theta_{2}W_{2}^{*})W_{2}D_{W_{2}}^{-1}-\Theta_{2}D_{W_{2}}]\Phi_{2}
D_{W_{1}^{*}}(I-\Theta_{1}W_{1}^{*})^{-1}\\
&=&-(I-\Theta_{2}W_{2}^{*})^{-1}[(I-\Theta_{2}W_{2}^{*})W_{2}-\Theta_{2}D_{W_{2}}^{2}]D_{W_{2}}^{-1}\Phi_{2}
D_{W_{1}^{*}}(I-\Theta_{1}W_{1}^{*})^{-1}\\
&=&-(I-\Theta_{2}W_{2}^{*})^{-1}[W_{2}-\Theta_{2}W_{2}^{*}W_{2}-\Theta_{2}+\Theta_{2}W_{2}^{*}W_{2})]D_{W_{2}}^{-1}\Phi_{2}
D_{W_{1}^{*}}(I-\Theta_{1}W_{1}^{*})^{-1}\\
&=&-(I-\Theta_{2}W_{2}^{*})^{-1}(W_{2}-\Theta_{2})D_{W_{2}}^{-1}\Phi_{2}
D_{W_{1}^{*}}(I-\Theta_{1}W_{1}^{*})^{-1}\\
&=&(I-\Theta_{2}W_{2}^{*})^{-1}\Theta_{2}(I-\Theta_{2}^{*}W_{2})D_{W_{2}}^{-1}\Phi_{2}
D_{W_{1}^{*}}(I-\Theta_{1}W_{1}^{*})^{-1}\\
&=&\Theta_{2}(I-W_{2}^{*}\Theta_{2})^{-1}(I-\Theta_{2}^{*}W_{2})D_{W_{2}}^{-1}\Phi_{2}
D_{W_{1}^{*}}(I-\Theta_{1}W_{1}^{*})^{-1}
\end{eqnarray*}
\begin{equation}\label{eq: for Theta 2}
D_{W_{2}^{*}}^{-1}\Theta_{2}^{\prime}\Phi_{2}
D_{W_{1}^{*}}(I-\Theta_{1}W_{1}^{*})^{-1}=\Theta_{2}(I-W_{2}^{*}\Theta_{2})^{-1}(I-\Theta_{2}^{*}W_{2})D_{W_{2}}^{-1}\Phi_{2}
D_{W_{1}^{*}}(I-\Theta_{1}W_{1}^{*})^{-1}.
\end{equation}
By (\ref{Eq: for Theta 1}) and (\ref{eq: for Theta 2}) we get
\begin{eqnarray*}
\Phi
&=&[\Theta_{1}(I-W_{1}^{*}\Theta_{1})^{-1}D_{W_{1}}\Phi_{1}D_{W_{2}^{*}}^{-1}]^{*}+\Theta_{2}(I-W_{2}^{*}\Theta_{2})^{-1}(I-\Theta_{2}^{*}W_{2})D_{W_{2}}^{-1}\Phi_{2}
D_{W_{1}^{*}}(I-\Theta_{1}W_{1}^{*})^{-1}\\
&-&\Theta_{2}[W_{2}^{*}D_{W_{2}^{*}}^{-1}(\Theta_{1}^{\prime}\Phi_{1})^{*}D_{W_{1}^{*}}(I-\Theta_{1}W_{1}^{*})^{-1}
+W_{2}^{*}D_{W_{2}^{*}}^{-1}\Theta_{2}^{\prime}\Phi_{2}
D_{W_{1}^{*}}(I-\Theta_{1}W_{1}^{*})^{-1}]\\
&=&[\Theta_{1}(I-W_{1}^{*}\Theta_{1})^{-1}D_{W_{1}}\Phi_{1}D_{W_{2}^{*}}^{-1}]^{*}+\Theta_{2}[(I-W_{2}^{*}\Theta_{2})^{-1}(I-\Theta_{2}^{*}W_{2})D_{W_{2}}^{-1}\Phi_{2}
D_{W_{1}^{*}}(I-\Theta_{1}W_{1}^{*})^{-1}\\
&-&(W_{2}^{*}D_{W_{2}^{*}}^{-1}(\Theta_{1}^{\prime}\Phi_{1})^{*}D_{W_{1}^{*}}(I-\Theta_{1}W_{1}^{*})^{-1}
+W_{2}^{*}D_{W_{2}^{*}}^{-1}\Theta_{2}^{\prime}\Phi_{2}
D_{W_{1}^{*}}(I-\Theta_{1}W_{1}^{*})^{-1})]\\
&=&[\Theta_{1}(I-W_{1}^{*}\Theta_{1})^{-1}D_{W_{1}}\Phi_{1}D_{W_{2}^{*}}^{-1}]^{*}
+\Theta_{2}[(I-W_{2}^{*}\Theta_{2})^{-1}(I-\Theta_{2}^{*}W_{2})D_{W_{2}}^{-1}\Phi_{2}\\
&-&W_{2}^{*}D_{W_{2}^{*}}^{-1}(\Theta_{1}^{\prime}\Phi_{1})^{*}
-W_{2}^{*}D_{W_{2}^{*}}^{-1}\Theta_{2}^{\prime}\Phi_{2}]D_{W_{1}^{*}}(I-\Theta_{1}W_{1}^{*})^{-1}\\
&=&[\Theta_{1}\Phi_{1}]^{*}+\Theta_{2}\Phi_{2},
\end{eqnarray*}
where $$\Phi_{1}=(I-W_{1}^{*}\Theta_{1})^{-1}D_{W_{1}}\Phi_{1}D_{W_{2}^{*}}^{-1}$$
and
\begin{eqnarray*}
\Phi_{2}
&=&[(I-W_{2}^{*}\Theta_{2})^{-1}(I-\Theta_{2}^{*}W_{2})D_{W_{2}}^{-1}\Phi_{2}
-W_{2}^{*}D_{W_{2}^{*}}^{-1}(\Theta_{1}^{\prime}\Phi_{1})^{*}
-W_{2}^{*}D_{W_{2}^{*}}^{-1}\Theta_{2}^{\prime}\Phi_{2}]D_{W_{1}^{*}}(I-\Theta_{1}W_{1}^{*})^{-1}
\end{eqnarray*}
\end{proof}
As a corollary, we show that every asymmetric matrix valued truncated Toeplitz operators has a symbol in certain class denoted by $\mathcal{M}_{\Theta_{1}}$ and $\mathcal{M}_{\Theta_{2}}$.
\begin{corollary}\label{Eq: Symbol is not unique}
For any $A\in \mathcal{T}(\Theta_{1},\Theta_{2})$ there exist $\Psi_{1}\in \mathcal{M}_{\Theta_{1}} $ and $\Psi_{2}\in \mathcal{M}_{\Theta_{2}}$ such that $A=A_{\Psi_{1}+\Psi_{2}^{*}}$. If $\Psi_{1}, \Psi_{2}$ is one such pair then the other such pair is $\Psi_{1}^{\prime}=\Psi_{1}+k_{0}^{\Theta_{2}}X$ and $\Psi_{2}^{\prime}=\Psi_{2}-X^{*}k_{0}^{\Theta_{1}}$ such that $A=A_{\Psi_{1}^{\prime}+\Psi_{2}^{\prime*}}$.
\end{corollary}
\begin{proof}
The first assertion follows by using Theorem \ref{Eq: main Theorem}. For the second part, consider
\begin{eqnarray*}
A_{\Psi_{1}^{\prime}+\Psi_{2}^{\prime*}}^{\Theta_{1}, \Theta_{2}}
&=&A_{\Psi_{1}+k_{0}^{\Theta_{2}}X+(\Psi_{2}-X^{*}k_{0}^{\Theta_{1}})^{*}}^{\Theta_{1}, \Theta_{2}}\\
&=&A_{\Psi_{1}+\Psi_{2}^{*}+k_{0}^{\Theta_{2}}X-(k_{0}^{\Theta_{1}})^{*}X}^{\Theta_{1}, \Theta_{2}}\\
&=&A_{\Psi_{1}+\Psi_{2}^{*}}^{\Theta_{1}, \Theta2_{}}+A_{k_{0}^{\Theta_{2}}X}^{\Theta_{1}, \Theta_{2}}-A_{(k_{0}^{\Theta_{1}})^{*}X}^{\Theta_{1}, \Theta_{2}}.
\end{eqnarray*}
Now consider
$$A_{k_{0}^{\Theta_{2}}X}^{\Theta_{1}, \Theta_{2}}f=P_{\Theta_{2}}(k_{0}^{\Theta_{2}}Xf)=P_{\Theta_{2}}(Xf).$$
Since $\Theta_{1}^{*}f\perp K_{\Theta_{1}}$ for all $f\in K_{\Theta_{1}}$, we get
$$A_{(k_{0}^{\Theta_{1}})^{*}X}^{\Theta_{1}, \Theta_{2}}f=P_{\Theta_{2}}((k_{0}^{\Theta_{1}})^{*}Xf)=P_{\Theta_{2}}((I-\Theta_{1}(0)\Theta_{1}(z)^{*})Xf)
=P_{\Theta_{2}}(Xf)-P_{\Theta_{2}}(\Theta_{1}(0)\Theta_{1}(z)^{*}Xf)=P_{\Theta_{2}}(Xf).$$
It follows that
$$A_{k_{0}^{\Theta_{2}}X}^{\Theta_{1}, \Theta_{2}}-A_{(k_{0}^{\Theta_{1}})^{*}X}^{\Theta_{1}, \Theta_{2}}=0,$$
and $$A_{\Psi_{1}^{\prime}+\Psi_{2}^{\prime*}}^{\Theta_{1}, \Theta_{2}}=A=A_{\Psi_{1}+\Psi_{2}^{*}}^{\Theta_{1}, \Theta2_{}}.$$
\end{proof}
It is known that the model space $K_{\Theta_{1}}$ (respectively $K_{\Theta_{2}}$) is finite dimensional if and only if $\Theta_{1}$ (respectively $\Theta_{2}$) is finite Blaschke-Potapov product (see, for instance \cite{Pe}, Chapter 2). In this case we use the previous corollary to obtain the dimension of the space of $\mathcal{T}(\Theta_{1}, \Theta_{2})$.
\begin{corollary}
If $dim K_{\Theta_{1}}=m$ and $dim K_{\Theta_{2}}=n$, then $dim\mathcal{T}(\Theta_{1}, \Theta_{2})=m^{d}+n^{d}-d^{2}$.
\end{corollary}
\begin{proof}
The proof is similar to Corollary 6.5 given in \cite{RDK}. For the sake of the reader we proof it analogously.\par
It is immediate that $ dim \mathcal{M}_{\Theta_{1}}=(dim K_{\Theta_{1}})^{d}=m^{d}$ and $ dim \mathcal{M}_{\Theta_{2}}=(dim K_{\Theta_{2}})^{d}=n^{d}$. Consider the map $T:\mathcal{M}_{\Theta_{1}}\times \mathcal{M}_{\Theta_{2}}\to \mathcal{T}(\Theta_{1}, \Theta_{2})$ defined by
$$T(\Psi_{1}+\Psi_{2})=A_{\Psi_{1}+\Psi_{2}^{*}}.$$
According to Corollary  \ref{Eq: Symbol is not unique}, $T$ is onto, while $A^{\Theta_{1},\Theta_{2}}_{k_{0}^{\Theta_{2}}X-(X^{*}k_{0}^{\Theta_{1}})^{*}}=0$, therefore we have
$$ker T={\{k_{0}^{\Theta_{2}}X-(X^{*}k_{0}^{\Theta_{1}})^{*}: X\in \mathcal{L}(E)}\}.$$
The proof completes by noting that $dim (\mathcal{M}_{\Theta_{1}}\times \mathcal{M}_{\Theta_{2}})=m^{d}+n^{d}$ and $dim ~kerT=dim \mathcal{L}(E)=d^{2}$.
\end{proof}


{\bf Acknowledgements:}

\end{document}